\documentclass[12pt]{article}
\usepackage{amssymb, amsmath}
\usepackage{amsthm}
\usepackage{mathrsfs}
\usepackage[margin=1in]{geometry}
\usepackage{enumerate}
\allowdisplaybreaks[1]
\numberwithin{equation}{section}

\newtheorem{thm}{Theorem}[section]
\newtheorem{cor}{Corollary}[section]

\newtheorem{rem}{Remark}[section]

\newtheorem{pro}{Proposition}[section]
\newtheorem{lemma}{Lemma}[section]
\usepackage{graphicx}
\usepackage{caption}
\usepackage{subcaption}

\begin{document}
\markboth{R. Rajkumar and P. Devi}{}
\title{Inclusion graph of subgroups of a group}

\author{P. Devi\footnote{e-mail: {\tt pdevigri@gmail.com}},\ \ \
	R. Rajkumar\footnote{e-mail: {\tt rrajmaths@yahoo.co.in}}\\
	{\footnotesize Department of Mathematics, The Gandhigram Rural Institute -- Deemed University,}\\ \footnotesize{Gandhigram -- 624 302, Tamil Nadu, India.}\\[3mm]
}
\date{}
\maketitle
\begin{abstract}
For a finite group $G$, we define the \textit{inclusion graph of subgroups of $G$}, denoted by $\mathcal I(G)$, is a graph having all the proper subgroups of 
$G$ as its vertices and two distinct vertices $H$ and $K$  in $\mathcal I(G)$ are adjacent if and only if either $H \subset K$ or $K \subset H$. 
 In this paper, we classify the finite groups whose inclusion graph of subgroups is one of complete, bipartite, tree, star, path,  cycle, disconnected,  claw-free. Also we classify the finite abelian groups whose inclusion graph of subgroups is planar. For any given finite group, we estimate the clique number, chromatic number, girth of its inclusion graph of subgroups and for a finite abelian group, we estimate the diameter of its inclusion graph of subgroups. Among the other results we show that some groups can be determined by their inclusion graph of subgroups
 \paragraph{Keywords:}Inclusion graph of subgroups, bipartite graph, clique number,  girth, diameter, planar graph.

\end{abstract}



\section{Introduction} \label{sec:int}
The properties of an algebraic structure can be investigated in several ways. One of the ways is by associating a suitable graph to that algebraic structure and analyzing the properties of the associated graph by using graph theoretic methods. The subgroup lattice and subgroup graph of a group are well known graphs associated with a group (cf. \cite{boh-reid}, \cite{smith1}, \cite{smith}, \cite{star}). The intersection graph of subgroups of a group is another interesting graph associated with a group (cf. \cite{akbari_2}, \cite{csak}, \cite{Shen}). For a group $G$, the intersection graph of subgroups of $G$, denoted by $\mathscr{I}(G)$, is a graph having all the proper subgroups of $G$ as  its vertices and two distinct vertices are adjacent if and only if  the corresponding subgroups intersects non-trivially. In \cite{akbaris}, S. Akbari et al assigned a graph to ideals of a ring as follows: 
For a ring $R$ with unity, the inclusion ideal graph of $R$ is a graph whose vertices are
all non-trivial left ideals of $R$ and two distinct left ideals $I$ and $J$ are adjacent if and only if $I \subset J $ or $J \subset I$.

 Motivated by this, in this paper,  we define the following: For a finite group $G$, the \textit{inclusion graph of subgroups of $G$}, denoted by $\mathcal I(G)$, is a graph having all the proper subgroups of 
$G$ as its vertices and two distinct vertices $H$ and $K$  in $\mathcal I(G)$ are adjacent if and only if  $H \subset K$ or $K \subset H$. 


For a group $G$, its subgroup lattice is denoted by $L(G)$. The \textit{height} of $L(G)$ is the length of the longest chain in $L(G)$ through partial order from greatest to least. We denote the order of an element $a \in \mathbb Z_n$ by $\text{ord}_n(a)$ . The number of Sylow $p$-subgroups of a group $G$ is denoted by $n_p(G)$. 

Now we recall some basic definitions and notations of graph theory. We use the standard terminology of graphs (e.g., see~\cite{harary}). Let $G$ be a simple graph with vertex set $V(G)$ and edge set $E(G)$. $G$ is said to be
\emph{$k$-partite}  if $V(G)$ can be partition into $k$ disjoint subsets $V_i$, $i=1,2, \ldots ,k$ such that every edge joins a vertex of $V_i$ to a vertex of $V_j$, $i\neq j$. A $k$-partite graph is said to be a \emph{complete $k$-partite} if every vertex in each partition is adjacent with all the vertices in the remaining partitions and is denoted by $K_{m_1, m_2, \ldots , m_k}$, 
where $m_i=|V_i|$, $i=1,\ldots, k$. The graph $K_{1,m}$ is called a \emph{star} and the graph $K_{1,3}$ is called a \emph{claw}. A graph in which any two distinct vertices are adjacent 
is said to be \emph{complete}. A graph whose edge set is empty is said to be \emph{totally disconnected}.  Two graphs $G_1$, $G_2$ are isomorphic if there exists 
a bijection from $V(G_1)$ to $V(G_2)$ preserving the adjacency.  A \textit{path} connecting two vertices $u$ and $v$ in $G$ is a finite sequence  $(u=) v_0, v_1, \ldots, v_n (=v)$ of distinct vertices (except,
possibly, $u$ and $v$) such that $u_i$ is adjacent to $u_{i+1}$ for all $i=0, 1, \ldots , n-1$.
A path is a \emph{cycle} if $u=v$. The length of a path or a cycle is the number of edges in it. A path or a cycle of length $n$ is denoted by $P_n$ or $C_n$, respectively. A graph $G$ is said to be \emph{connected} if any two vertices are connected by a path; otherwise $G$ is said to be \emph{disconnected}. A connected graph with out a cycle is called a \emph{tree}. For a connected graph $G$, its \emph{diameter}, denoted by $diam(G)$, is the maximum of length of a 
shortest path of any two vertices. If $G$ is disconnected, then we define $diam(G)=\infty$. 
The \emph{girth} of a graph $G$, denoted by $girth(G)$, is the length of a shortest cycle in $G$, if it exist; otherwise, we define $girth(G)=\infty$. A \emph{clique}  is a set of vertices in $G$ such that any two
are adjacent. The \emph{clique number} $\omega (G)$ of $G$ is the cardinality of a largest  clique in $G$. The \emph{chromatic number}
$\chi (G)$ of $G$ is the smallest number of colors needed to
color the vertices of $G$ such that no two adjacent vertices gets the same color.
 A graph is said to be \emph{planar} if it can be drawn on a plane such that no two edges intersect, except, possibly at their end vertices.
We define a graph $G$ to be \emph{$X$-free} if it does not contain a subgraph isomorphic to a given graph $X$. $\overline{G}$ denotes the complement of a graph $G$. For two graphs $G$ and $H$, $G\cup H$ denotes the disjoint union of $G$ and $H$.

In this paper, we classify all the finite groups whose inclusion graph of subgroups is one of complete, bipartite, tree, star, path,  cycle, claw-free, disconnected (cf. Theorems~\ref{inclusion graph 112},~\ref{inclusion graph 3},~\ref{incl 121}, ~\ref{inclusion graphs 7}, Corollaries~\ref{inclusion graphs 3},~\ref{inclusion graphs 4}). Also we give the classification of finite abelian groups whose inclusion graph of subgroups is planar (cf. Theorem~\ref{inclusion graphs 10}). For any given finite group, we estimate the clique number, chromatic number, girth of its inclusion graph of subgroups (cf. Theorem~\ref{inclusion graphs 7}, Corollary~\ref{inclusion graph c4}) and for a finite abelian group, we estimate the diameter of its inclusion graph of subgroups (cf. Theorem~\ref{inclusion graphs 11}). Moreover, we show that some groups can be determined by their inclusion graph of subgroups (cf. Corollary~\ref{inclusion graphs 15}). In this sequel, we show some interesting connections between the inclusion graph of subgroups of a group,  its subgroup lattice and its intersection graph of subgroups (cf. Theorems~\ref{incl l2},~\ref{inclusion graph 1},~\ref{inclusion graph 3},~\ref{inclusion graph 2},~\ref{inclusion graphs 41} and Corollary~\ref{inclusion graph c4}).

%

\section{Main results}\label{sec: inclbas}

\begin{thm}\label{inclusion graphs 14}
 Let $G_1$ and $G_2$ be groups. If $G_1\cong G_2$, then $\mathcal I(G_1)\cong \mathcal I(G_2)$.
\end{thm}
\begin{proof}
If $\phi$ is a group isomorphism from $G_1$ into $G_2$, then the map $\psi:V(\mathcal I(G_1))\rightarrow V(\mathcal I(G_2))$ by $\psi(H)=\phi(H)$, for all $H\in V(\mathcal I(G_1))$ is a graph 
isomorphism. 
\end{proof}

\begin{rem}\label{incl n1}
The converse of the above theorem is not true. For example, let $G_1= \mathbb Z_3\times \mathbb Z_3$ and $G_2= S_3$, then it is easy to see that
 $\mathcal I(G_1)\cong \overline{K}_4\cong \mathcal I(G_2)$, but $G_1 \ncong G_2$.
\end{rem}

\begin{thm}\label{incl l2}
Let $G_1$ and $G_2$ be groups. If $L(G_1)\cong L(G_2)$, then $\mathcal I(G_1)\cong \mathcal I(G_2)$.
\end{thm}
\begin{proof}
Let $\phi$ be a lattice isomorphism from $L(G_1)$ into $L(G_2)$. Define a map $\psi:V(\mathcal I(G_1)) \rightarrow V(\mathcal I(G_2))$ by $\psi(H)=\phi(H)$, for all $H\in V(\mathcal I(G_1))$. 
Since $\phi$ is bijective, so is $\psi$.
Suppose $H_1$ and $H_2$ are adjacent in $\mathcal I(G_1)$, then either $H_1\subset H_2$ or $H_2\subset H_1$. Since $\phi$ is a lattice isomorphism, 
so it preserves the meet and order and so 
either $\phi(H_1)\subset \phi(H_2)$ or $\phi(H_2)\subset \phi(H_1)$. It follows that $\psi(H_1)$ and $\psi(H_2)$ are adjacent in $\mathcal I(G_2)$. By following a similar argument as above, it is easy to see that if $\psi(H_1)$ and $\psi(H_2)$ are adjacent in $\mathcal I(G_2)$, then $H_1$ and $H_2$ are adjacent in $\mathcal I(G_1)$. Thus $\psi$ is a graph isomorphism.  Hence the proof.
\end{proof}

\begin{thm}\label{inclusion graph 13}
	Let $G$ be any group and $N$ be a subgroup of $G$. Then $\mathcal I(N)$ is a subgraph of $\mathcal I(G)$. In addition, $N$ is a normal subgroup,
	then $\mathcal I(G/N)$ is isomorphic (as a graph) to a subgraph of $\mathcal I(G)$.
\end{thm}
\begin{proof}
	The first result is obviously true. Any subgroup of $G/N$ is of the form $H/N$, where $H$ is a subgroup of $G$ containing $N$. Here two proper subgroups $H/N$, $K/N$ 
	are adjacent in $\mathcal I(G/N)$ if and only if either $H/N \subset K/N$ or $K/N \subset H/N$. This implies that either $H \subset K$ or $K \subset H$ and so $H$ and $K$ are 	adjacent in $\mathcal I(G)$. This completes the proof.
\end{proof}

\begin{thm}\label{inclusion graph 112}
	Let $G$ be a finite group. Then $\mathcal I(G)$ is complete if and only if $G\cong \mathbb Z_{p^\alpha}$, where $p$ is a prime and $\alpha >1$. Moreover, $\mathcal I(\mathbb Z_{p^\alpha})\cong K_{\alpha-1}$.
\end{thm}
\begin{proof}
If $G\cong \mathbb Z_{p^\alpha}$,  where $p$ is a prime and $\alpha >1$, then $L(G)$ is a chain of length $\alpha$ and so $\mathcal I(G)\cong K_{\alpha-1}$. 
If $G\ncong \mathbb Z_{p^\alpha}$, then $L(G)$ is not a chain and so there exists at least 
two subgroups  $H_1$ and $H_2$ of $G$ such that $H_1 \nsubseteq H_2$ and $H_2 \nsubseteq H_1$. It follows that $\mathcal I(G)$ is not complete.
\end{proof}


\begin{thm}\label{inclusion graph 1}
Let $G$ be a finite group. Then $\mathcal I(G)$ is $(k-1)$-partite, where $k$ is the height of $L(G)$.
\end{thm}
\begin{proof}
Let $\mathcal M_1$ be the set of all maximal subgroups of $G$ and for each $i = 2, 3, \ldots, k-1$, let $\mathcal M_i$ be the set of all maximal subgroups of the subgroups in $\mathcal  M_{i-1}$. Then $\{\mathcal M_i\}_{i=1}^{k-1}$ is a partition of the vertex set of $\mathcal I(G)$. Also no two vertices in
a same partition are adjacent in $\mathcal I(G)$. Moreover, $k-1$ is the minimal number such that a $k$-partition
of the vertex set of $\mathcal I(G)$ is having this property, since  the height of $L(G)$ is $k$. It follows that $\mathcal I(G)$ is $(k-1)$-partite.
\end{proof}

The next result is an immediate consequence of Theorem~\ref{inclusion graph 1}, the definition of the clique number and chromatic number of a graph.

\begin{cor}\label{inclusion graph c4}
	Let $G$ be a finite group. Then $\omega(\mathcal I(G))=k-1=\chi(\mathcal I(G))$, where $k$ is the height of $L(G)$.
\end{cor}

The next two results are immediate consequences of Theorem~\ref{inclusion graph 1} and the definition of the subgroup lattice of a group.
\begin{thm}\label{inclusion graph 3}
	Let $G$ be a group. Then the following are equivalent.
	\begin{enumerate}[\normalfont (i)]
			
		\item $\mathcal I(G)$ is totally disconnected;	
			
			\item every proper subgroups of $G$ is of prime order;
		
		\item height of $L(G)$ is  2.		
			
	\end{enumerate} 
\end{thm}

\begin{thm}\label{inclusion graph 2}
	Let $G$ be a group and $e$ be its identity element. Then the following are equivalent.
	\begin{enumerate}[\normalfont (i)]
			\item $\mathcal I(G)$ is bipartite;
		\item $L(G)-\{G, e\}\cong \mathcal I(G)$;
	
		\item height of $L(G)$ is either 2 or 3.
	\end{enumerate} 
\end{thm}

\begin{cor}\label{inclusion graphs 3}
	Let $G$ be a finite group and $p$, $q$, $r$ be distinct primes. Then
	\begin{enumerate}[\normalfont (1)]
	\item $\mathcal I(G)$ is totally disconnected if and only if $G$ is 
	one of $\mathbb Z_{p^2}$, $\mathbb Z_p\times \mathbb Z_p$, $\mathbb Z_{pq}$, $\mathbb Z_q\rtimes \mathbb Z_p$;
	
	\item $\mathcal I(G)$ is bipartite if and only if $|G|$ is one of $p^2$, $pq$, $p^3$, $p^2q$ or $pqr$.
	\end{enumerate} 
\end{cor}

Consider the semi-direct product $\mathbb Z_q \rtimes_{t} \mathbb Z_{p^{\alpha}} = \langle a,b~|~a^q= b^{p^{\alpha}}= 1, bab^{-1}= a^i,
{ord_{q}}(i)= p^t \rangle$, where $p$ and $q$ are distinct primes with $p^t~|~(q-1)$, $t \geq 0$. Then every semi-direct product $Z_q \rtimes Z_{p^{\alpha}}$
is  one of these types \cite[Lemma 2.12]{boh-reid}. Note that, here after we suppress the subscript when $t = 1$.

\begin{thm}\label{incl 121}
Let $G$ be a finite group and $p$, $q$, $r$ be distinct primes. Then 
\begin{enumerate}[\normalfont (1)]
	\item $\mathcal I(G)\cong C_n$ if and only if either $n=3$ and $G\cong \mathbb Z_{p^4}$ or $n=6$ and $G\cong \mathbb Z_{pqr}$;
	
\item $\mathcal I(G)$ is a tree if and only if $G$ is one of $\mathbb Z_{p^3}$, $\mathbb Z_{p^2q}$, $\mathbb Z_{p^2}\times \mathbb Z_p$, $Q_8$, 
$M_8$, $\mathbb Z_q\rtimes \mathbb Z_{p^2}$, $\mathbb Z_q\rtimes_2 \mathbb Z_{p^2}$ or $\mathbb Z_{p^2}\rtimes\mathbb Z_q=\langle a,b~|~a^{p^2}=b^q=1, bab^{-1}=a^i, \mbox{ord}_{p^2}(i)=q\rangle$ $(q~|(p-1))$;

\item $\mathcal I(G)$ is a star  if and only if $G$ is either $\mathbb Z_{p^3}$ or $Q_8$;
\item $\mathcal I(G)\cong P_n$ if and only if either $n=1$ and $G\cong\mathbb Z_{p^3}$ or $n=3$ and $G \cong \mathbb Z_{p^2q}$.

\end{enumerate}
\end{thm}
\begin{proof}
	
	First we claim that  $\mathcal I(G)\cong C_n$, where $n$ is odd if and only if $n=3$ and $G\cong \mathbb Z_{p^4}$.
	
	It is easy to see that  $G\cong \mathbb Z_{p^4}$ if and only if $\mathcal I(G)\cong C_3$. Now suppose that $\mathcal I(G)\cong C_n$, where $n$ is odd, $n \geq 5$.  Let this cycle be  $H_1-H_2-\cdots-H_n-H_1$. Now we have two possibilities: $H_1\subset H_2$, $H_n$ or $H_2$, $H_n\subset H_1$. 
	
\begin{enumerate}[\normalfont (i)]
	\item  Suppose that $H_1\subset H_2$, $H_n$. Since $H_2$ is adjacent to $H_3$, so either $H_2\subset H_3$ or $H_3\subset H_2$. If $H_2\subset H_3$, then $H_1, H_2, H_3$ forms $C_3$ as a subgraph of $\mathcal I(G)$, which is a contradiction and so we must have $H_3\subset H_2$. Next, since $H_3$ is adjacent to $H_4$, so either $H_3\subset H_4$ or $H_4\subset H_3$. If $H_4\subset H_3$, then $H_2,H_3,H_4$ forms $C_3$ as a subgraph in $\mathcal I(G)$, which is a contradiction and so we must have $H_3\subset H_4$. If we proceed like this, we get $H_i\subset H_{i-1}$, $H_{i+1}$, when $i$  is an  odd integer (suffixes taken modulo $n$) and so  $H_n\subset H_1$, which is not possible.
	
	\item  Suppose that $H_2$, $H_n\subset H_1$. 
Since $H_2$ is adjacent to $H_3$, so either $H_2\subset H_3$ or $H_3\subset H_2$. If $H_3\subset H_2$, then $H_1,H_2,H_3$ forms $C_3$ as a subgraph of $\mathcal I(G)$, which is a contradiction and so we must have and so we have $H_2\subset H_3$. Next, since $H_3$ is adjacent to $H_4$, so either $H_3\subset H_4$ or $H_4\subset H_3$. If $H_3\subset H_4$, then $H_2,H_3,H_4$ forms $C_3$ as a subgraph of $\mathcal I(G)$ and so we have  $H_4\subset H_3$. If we proceed like this, we get $H_i\subset H_{i-1}$, $H_{i+1}$, when $i$  is an  even integer (suffixes taken modulo $n$) and so we have $H_{n-1} \subset H_n$. This implies that $H_{n-1},H_n,H_1$ forms $C_3$ as a subgraph  of $\mathcal I(G)$, which is a contradiction.
\end{enumerate} 
Thus $\mathcal I(G)\ncong C_n$, where $n$ is odd, $n \geq 5$ and our claim is now proved.

Next, we start to prove the main theorem. Since every tree, star graph, path, even cycle are bipartite, so to classify the finite groups whose inclusion graph of subgroups is one of tree, star graph, path, even cycle, it is enough to consider the groups of order $p^2$, $pq$, $p^3$, $p^2q$, $pqr$, by Corollary~\ref{inclusion graphs 3}(2).

\noindent\textbf{Case 1:} Let $|G|=p^2$ or $pq$. By Theorem~\ref{inclusion graph 3}, $\mathcal I(G)$ is neither a tree nor a cycle.

\noindent\textbf{Case 2:} Let $|G|=p^3$. 
 Here we use the classification of groups of order $p^3$.
 \begin{enumerate}[\normalfont (i)]
\item If $G \cong \mathbb Z_{p^3}$, then by Theorem~\ref{inclusion graph 112}, $\mathcal I(G)\cong K_2$, which is a path but not a cycle.
 
\item If $G\cong \mathbb Z_{p^2}\times \mathbb Z_p$, then 
$\langle (1,0)\rangle$, $\langle (1,1)\rangle$,  $\ldots$, $\langle (1,p-1)\rangle$, $\langle (p,0), (0,1)\rangle$, 
$\langle (p,0)\rangle$, $\langle (p,1)\rangle$, $\ldots$, $\langle (p,p-1)\rangle$, $\langle (0,1)\rangle$ are the only proper 
subgroups of $G$. $\mathcal I(G)$ is shown in Figure~\ref{incl fig}(a), which is a tree, but none of a star, path, cycle.


\item If $G\cong \mathbb Z_p\times \mathbb Z_p\times \mathbb Z_p:=\langle a, b, c~|~ a^p=b^p=c^p=1, ab=bc, ac=ca, bc=cb\rangle$, then $\mathcal I(G)$ contains $C_6$ as a subgraph: $\langle a,b\rangle-\langle b\rangle-\langle b,c\rangle-\langle c\rangle-\langle a,c\rangle-\langle a\rangle-\langle a,b\rangle$. So $\mathcal I(G)$ is neither a tree nor a cycle.

\item If $G\cong Q_8=\langle a, b~|~a^4=b^4=1, a^2=b^2, ab=ba^{-1}\rangle$, then $\langle a\rangle$, $\langle b\rangle$, $\langle ab\rangle$, $\langle a^2\rangle$ are the only 
proper subgroups of $G$.
It follows that 
\begin{equation}\label{incle2}
\mathcal I(Q_8)\cong K_{1,3}, 
\end{equation}
which is a star; but neither a path nor a cycle. 

\item If $G\cong M_8=\langle a, b~|~a^4=b^2=1, ab=ba^{-1}\rangle$, then $\langle a\rangle$, 
$\langle a^2,b\rangle$, $\langle a^2,ab\rangle$, 
$\langle b\rangle$, $\langle a^2\rangle$, $\langle ab\rangle$, $\langle a^2b\rangle$, $\langle a^3b\rangle$ 
are the only proper subgroups of $G$. It follows that $\mathcal I(G)$ is as shown in Figure~\ref{incl fig}(b), 
which is a tree; but none of a star, path, cycle.

\item If $G\cong M_{p^3}$, $p > 2$,  then subgroup lattice of $M_{p^3}$ and $\mathbb Z_{p^2}\times \mathbb Z_p$ are isomorphic. It follows from Theorem~\ref{incl l2} that $\mathcal I(G)$ is as in~Figure~\ref{incl fig}(a), which is a tree, but none of a star, path, cycle. 

\item If $G\cong (\mathbb Z_p\times \mathbb Z_p)\rtimes \mathbb Z_p:=\langle a, b, c~|~a^p=b^p=c^p=1, ab=ba, ac=ca cbc^{-1}=ab\rangle$, then $\mathcal I(G)$ contains $C_6$ as a subgraph: $\langle a,b\rangle-\langle b\rangle-\langle b,c\rangle-\langle c\rangle-\langle a,c\rangle-\langle a\rangle-\langle a,b\rangle$. So $\mathcal I(G)$ is neither a tree nor cycle. 
\end{enumerate}

\noindent\textbf{Case 3:} Let $|G|=p^2q$.  Here we use the classification of groups of order $p^2q$ given in~\cite[pp. 76–-80]{burn}.

\noindent\textbf{Subcase 3a:} Let  $G$ be abelian.
\begin{enumerate}[\normalfont (i)]
	\item If  $G\cong \mathbb Z_{p^2q}$, then it is easy to see that
	\begin{equation}\label{e4}
	\mathcal I(Z_{p^2q})\cong P_3, 
	\end{equation}
	which is a path; but neither a star nor a cycle.
	
	\item If $G\cong \mathbb Z_{pq}\times \mathbb Z_p$, then  $H:=\mathbb Z_p\times \mathbb Z_p$, $A:=\mathbb Z_q\times \{e\}$, 
	$H_x:=\langle (1,x)\rangle$ $(x=0,1,\ldots,p-1)$, $H_p=\langle (0,1)\rangle$, $A_x:=H_xA$, are the only subgroups of $G$. It follows that $\mathcal I(G)$ is as shown in Figure~\ref{incl fig}(c), 
	which is neither a tree nor a cycle.
	
\end{enumerate}

\noindent\textbf{Subcase 3b:} Let $G$ be non-abelian.

\noindent \textbf{Subcase 3b(I):} Let $p<q$.
\begin{enumerate}[\normalfont (i)]
	\item If $p~\nmid (q-1)$, by Sylow's Theorem, there are no non-abelian groups.
	
	\item If $p~|~(q-1)$ but $p^2~\nmid~(q-1)$, then we have two groups. The first one is 
	$G_1\cong \mathbb Z_q\rtimes \mathbb Z_{p^2}=\langle a,b~|~a^q=b^{p^2}=1, bab^{-1}=a^i, \mbox{ord}_q(i)=p\rangle$. Here $\langle a,b\rangle$, 
	$\langle a^ib\rangle$ $(i=1,2,\ldots,q)$, $\langle ab^p\rangle$, $\langle b^p\rangle$ are the only proper subgroups of $G_1$. It follows that $\mathcal I(G_1)$ is  as shown in Figure~\ref{incl fig}(d), 
	which is a tree; but none of a star, path, cycle.
	
	The second one is $G_2=\langle a,b,c~|~a^q=b^p=c^p=1, bab^{-1}=a^i, ac=ca, bc=cb, \mbox{ord}_q(i)=p\rangle$. Then
	$\mathcal I(G)$ contains $C_6$ as a proper subgraph: $\langle a,b\rangle-\langle b\rangle-\langle b,c\rangle-\langle c\rangle-\langle a,c\rangle-\langle a\rangle-\langle a,b\rangle$ and so $\mathcal I(G_2)$ is neither a tree nor a cycle.

	\item If $p^2~|~(q-1)$, then we have both groups $G_1$ and $G_2$ mentioned in (ii) of this subcase, together with the group 
	$G_3=\mathbb Z_q\rtimes_2 \mathbb Z_{p^2}=\langle a,b~|~a^q=b^{p^2}=1, bab^{-1}=a^i, \mbox{ord}_q(i)=p^2\rangle$. In (ii) of this subcase, we already dealt with $G_1$, $G_2$. 
	By Sylow's Theorem, $G_3$ has a unique subgroup, say $H$ of order $q$, and has $q$ Sylow-$p$ subgroups of order $p^2$, say $H_i$, $i=1,2,\ldots,q$ ; 
	for each $i=1,2,\ldots,q$, $H_i$ has a unique subgroup $H_i'$ of order $p$ and $G_3$ has a unique subgroup, say $K$ of order $pq$. Here $H$ is a subgroup of 
	$K$; for each $i=1,2,\ldots,q$, $H_i'$ is a subgroup of $H_i$, $K$; for each $i=1,2,\ldots,q$, $H_i$ is a subgroup of $K$; for each $i=1,2,\ldots,q$, 
	no two $H_i$ contained in $H_j$ and $H_i$' contained in $H_j'$, for every $i\neq j$. It follows that $\mathcal I(G_3)$ is shown in Figure~\ref{incl fig}(e), 
	which is neither a tree nor a cycle.  
	
\end{enumerate}
%


 
\noindent \textbf{Subcase 3b(II):} Let $p>q$.

\begin{enumerate}[\normalfont (i)]
\item If $q~\nmid (p^2-1)$, then there is no non-abelian group in this case.

\item If $q~|~(p-1)$, then we have two groups. The first group is 
$G_4\cong \mathbb Z_{p^2}\rtimes\mathbb Z_q=\langle a,b~|~a^{p^2}=b^q=1, bab^{-1}=a^i, \mbox{ord}_{p^2}(i)=q\rangle$. By Sylow's Theorem, $G_4$ has a unique 
subgroup $H$ of order $p^2$; $H$ has a unique subgroup $H'$ of order $p$ and $p^2$ Sylow $q$-subgroups of order $q$, say $H_1$, $H_2$,$\ldots$, $H_{p^2}$  and 
$p$ subgroups of order $pq$, say $N_1$, $N_2$, $\ldots$, $N_p$; these are the only proper subgroups of $G_4$. 
It follows that $\mathcal I(G_4)$ is as shown in Figure~\ref{incl fig}(f), 
which is a tree but none of a star, path, cycle.

\begin{figure}
    \centering
    \begin{subfigure}[b]{0.3\textwidth}
        \includegraphics[width=\textwidth]{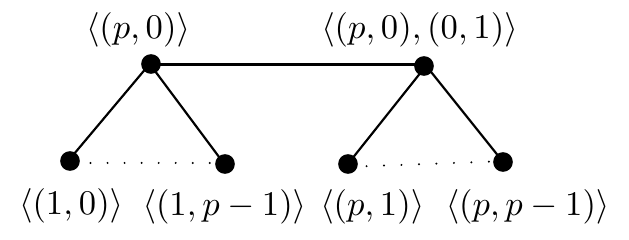}
        \caption{}
        \label{incl fig1}
    \end{subfigure}
            \begin{subfigure}[b]{0.3\textwidth}
        \includegraphics[width=\textwidth]{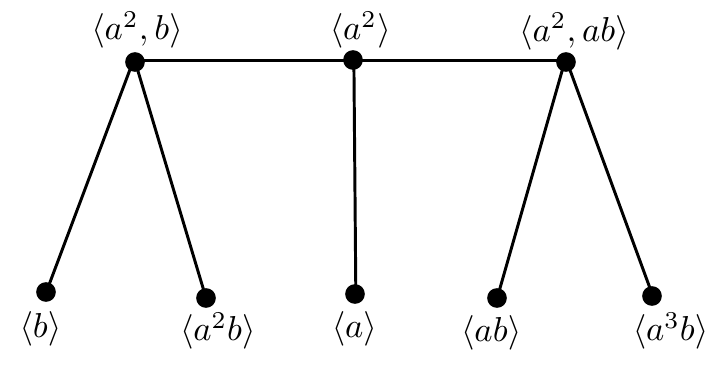}
        \caption{}
        \label{incl fig2}
    \end{subfigure}
       \begin{subfigure}[b]{0.3\textwidth}
        \includegraphics[width=\textwidth]{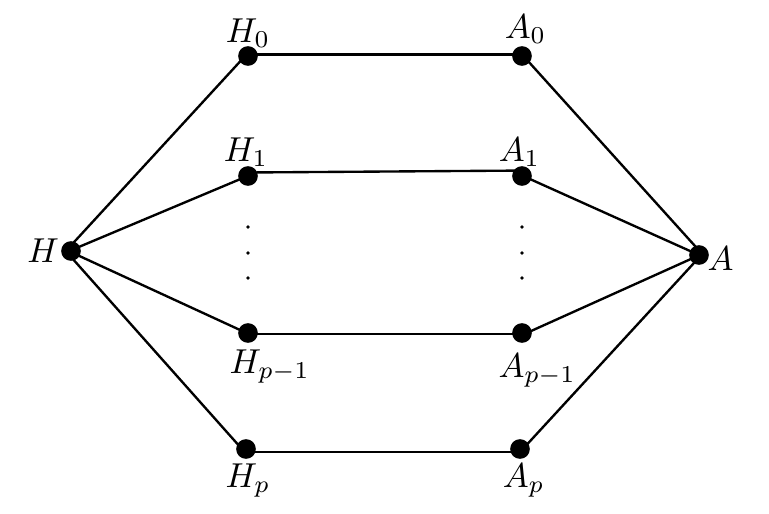}
        \caption{}
        \label{incl fig3}
    \end{subfigure}
    
    \begin{subfigure}[b]{0.18\textwidth}
            \includegraphics[width=\textwidth]{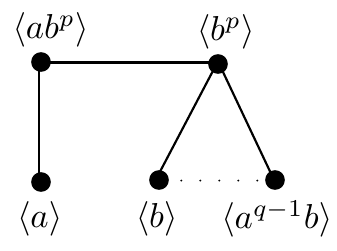}
            \caption{}
            \label{incl fig4}
        \end{subfigure}
        \begin{subfigure}[b]{0.2\textwidth}
                \includegraphics[width=\textwidth]{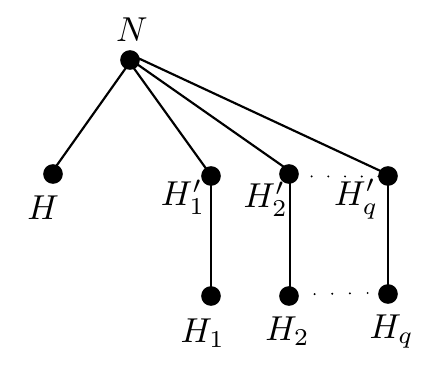}
                \caption{}
                \label{incl fig5}
            \end{subfigure}
            \begin{subfigure}[b]{0.3\textwidth}
                    \includegraphics[width=\textwidth]{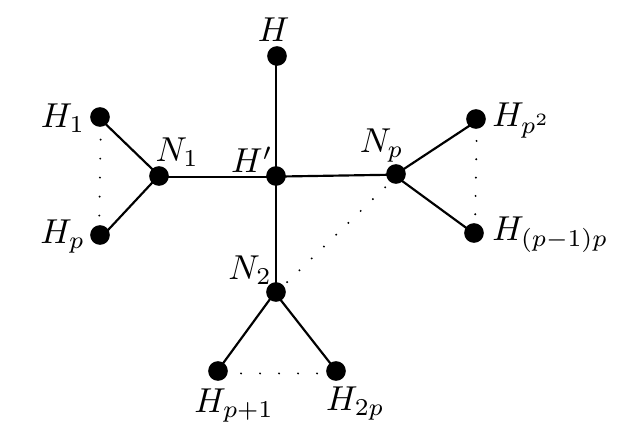}
                    \caption{}
                    \label{incl fig6}
                \end{subfigure}
    \caption{(a) $\mathcal I(\mathbb Z_{p^2}\times \mathbb Z_p)$, (b) $\mathcal I(M_8)$, (c) $\mathcal I(\mathbb Z_{pq}\times \mathbb Z_p)$, (d) $\mathcal I(\mathbb Z_q\rtimes \mathbb Z_{p^2})$, (e) $\mathcal I(\mathbb Z_q\rtimes_2 \mathbb Z_{p^2})$, (f) $\mathcal I(\mathbb Z_{p^2}\rtimes \mathbb Z_q)$}\label{incl fig}
\end{figure}


Second, we have a family of groups $\langle a,b,c~|~a^p=b^p=c^q=1, cac^{-1}=a^i, cbc^{-1}=b^{i^t}, ab=ba, \mbox{ord}_p(i)=q\rangle$. As $q>2$, there are $(q+3)/2$ 
isomorphism types in this family (one for $t=0$ and one for each pair $\{x,x^{-1}\}$ in $F_p^x$). We will refer to all these groups as $G_{5(t)}$ of order 
$p^2q$. Then
$\mathcal I(G_{(5t)})$ contains $C_6$ as a subgraph: $\langle a,b\rangle-\langle b\rangle-\langle b,c\rangle-\langle c\rangle-\langle a,c\rangle-\langle a\rangle-\langle a,b\rangle$. So $\mathcal I(G_{5(t)})$ is neither a tree nor a cycle. 

\item If $q~|~(p+1)$, then we have only one group of order $p^2q$, given by $G_6\cong (\mathbb Z_{p}\times \mathbb Z_p)\rtimes \mathbb Z_q=\langle a,b,c
~|~a^p=b^p=c^q=1, ab=ba, cac^{-1}=a^ib^j, cbc^{-1}=a^kb^l\rangle$, where $\bigl(\begin{smallmatrix}
i & j\\ k & l
\end{smallmatrix} \bigr)$ has order $q$ in $GL_2(p)$. 
Then $G_6$ has unique subgroup of order $p^2$, let it be $H$; $p+1$ subgroups of order $p$, let them be $H_i$, $i=1$, 2, $\ldots$, $p$; $p^2$ subgroups of order $q$; these are the only proper subgroups of $G_6$. It follows that
\begin{equation}\label{e17}
	\mathcal I(G_6)\cong K_{1, p+1}\cup \overline{K}_{p^2}, 
\end{equation}

 which is neither a tree nor a cycle.
\end{enumerate}
Note that if $(p,q)=(2,3)$, then subcases 3b(I) and 3b(II) are not mutually exclusive. Up to isomorphism, there are three non-abelian groups of order 12: 
$\mathbb Z_3\rtimes \mathbb Z_4$, $D_{12}$ and $A_4$. We have already dealt with $\mathbb Z_3\rtimes \mathbb Z_4$, $D_{12}$ in (ii) of Subcase 3b(I). But for $A_4$, we 
cannot use the same argument as in (iii) of Subcase 3b(II). Now  $A_4$ has unique subgroup of order 4, say $H$ ; three subgroups of order 
2, say $H_1$, $H_2$, $H_3$  and four subgroups  of order 3, say $N_1$, $N_2$, $N_3$, $N_4$. Here for each $i=1,2,3$, $H_i$ is a subgroup of $H$; for each $i=1,2,3,4$, $N_i$ is not 
contained in any other subgroups of $A_4$; no two remaining subgroups are adjacent in $\mathcal I(A_4)$. Therefore, 
\begin{equation}\label{e8}
\mathcal I(A_4)\cong K_{1,3}\cup \overline{K}_4,
\end{equation}
which is neither a tree nor a  cycle.

\noindent\textbf{Case 4:} Let $|G|=pqr$. 

If $G\cong \mathbb Z_{pqr}$, then let $H_1$, $H_2$, $H_3$, $H_4$, $H_5$, $H_6$ be the   subgroups of $G$ of 
orders $p$, $q$, $r$, $pq$, $pr$, $qr$, respectively. Here $H_1$ is a subgroup of $H_4$, $H_5$; $H_2$ is a subgroup of $H_4$, $H_6$; $H_3$ 
is a subgroup of $H_5$, $H_6$. So it turns out that $\mathcal I(G)\cong C_6$, which is not a tree.

If $G$ is non-abelian, then $G$ has a Sylow basis containing Sylow $p$, $q$, $r$-subgroups, let them be $H_1$, $H_2$, $H_3$, respectively and so $H_1H_2$, $H_2H_3$, $H_1H_3$ are proper subgroups of $G$. It follows that $\mathcal I(G)$ contains $C_6$ as a proper subgraph: $H_1-H_1H_2-H_2-H_2H_3-H_3-H_1H_3-H_1$. So $\mathcal I(G)$ neither a tree nor a cycle.

Proof follows by combining all the above cases together.
\end{proof}

In the next result, we characterize some groups by using their inclusion graph of subgroups.
 
\begin{cor}\label{inclusion graphs 15} 
Let $G$ be a group.
	\
	\begin{enumerate}[{\normalfont (1)}]
		\item If $\mathcal I(G)\cong\mathcal I(Q_8)$, then $G\cong Q_8$;
		\item If $\mathcal I(G)\cong\mathcal I(M_8)$, then $G\cong M_8$.
		\item If $\mathcal I(G)\cong\mathcal I(\mathbb Z_{9}\rtimes \mathbb Z_2)$, then $G\cong \mathbb Z_{9}\rtimes \mathbb Z_2$
	\end{enumerate}
\end{cor}
\begin{proof}
	By Theorem~\ref{incl 121}(2), $\mathcal I(Q_8)$, $\mathcal I(M_8)$, $\mathcal I(\mathbb Z_{9}\rtimes \mathbb Z_2)$ are trees and by \eqref{incle2}, Figures~\ref{incl fig}(b) and \ref{incl fig}(f), these trees uniquely determines the corresponding group. Hence the proof.
\end{proof}


\begin{thm}\label{inclusion graphs 41}
Let $G$ be a group. Then $\mathcal I(G)$ is connected if and only if $\mathscr{I}(G)$ is connected.
\end{thm}
\begin{proof}
Since $\mathcal I(G)$ is a spanning subgraph of $\mathscr{I}(G)$, so if $\mathscr{I}(G)$ is connected, then so is $\mathscr{I}(G)$. Now let $H$, $K$ be two adjacent vertices in $\mathscr{I}(G)$. Then exactly one of the following holds: $H \subset K$, $K \subset H$, neither $H \subset K$ nor $K \subset H$ but $|H \cap K |>1$.   If one of the first two possibilities holds, then $H$ and $K$ are adjacent in $\mathcal I(G)$. If the third condition holds, then $H-H\cap K-K$ is a path in $\mathcal I(G)$. So it follows that if $\mathscr{I}(G)$ is connected, then so is $\mathcal I(G)$.
\end{proof}

In \cite{Shen}, Rulin Shen et al classified all the finite groups whose intersection graphs of subgroups are disconnected. So as a consequence of Theorem~\ref{inclusion graphs 41}, we have the following result.

\begin{cor}\label{inclusion graphs 4}
	Let $G$ be a finite group. Then $\mathcal I(G)$ is disconnected if and only if $G$ is one of $\mathbb Z_p \times \mathbb Z_q$, 
	where both $p$, $q$ are primes, or a Frobenius group whose complement is a prime order group and the kernel is a minimal normal subgroup.
\end{cor}

\begin{thm}\label{inclusion graphs 10}
 Let $G$ be a finite abelian group and $p, q, r$ be distinct primes. Then $\mathcal I(G)$ is planar if and only if $G$ is one of  $\mathbb Z_{p^\alpha}(\alpha=2,3,4,5)$,
 $\mathbb Z_{p^\alpha q}(\alpha = 1,2,3, 4)$, $\mathbb Z_{p^2 q^2}$,
  $\mathbb Z_{pqr}$, $\mathbb Z_{p^2qr}$,
  $\mathbb Z_{pqrs}$, $\mathbb Z_p\times \mathbb Z_p$, $\mathbb Z_{p^2}\times \mathbb Z_p$ or $\mathbb Z_{pq}\times \mathbb Z_p$.
\end{thm}
\begin{proof}
 Let $|G|=p_1^{\alpha_1}p_2^{\alpha_2}\ldots p_k^{\alpha_k}$, where $p_i$'s are distinct primes and $\alpha_i\geq 1$.
 
  \noindent \textbf{Case 1:} Let $G$ be abelian.
  
 \noindent \textbf{Subcase 1a:} $k=1$. By Theorem~\ref{inclusion graph 112}, $\mathcal I(G)$ is planar if and only if $\alpha = 2,3,4,5$.
 
 \noindent \textbf{Subcase 1a:} $k=2$.  
 \begin{enumerate}[\normalfont (i)] 	
 	\item  If either $\alpha_1\geq 5$ or $\alpha_2\geq 5$, then $L(G)$ has a chain of length at least four and so $\mathcal I(G)$ contains $K_5$ as a subgraph. This implies that $\mathcal I(G)$ is non-planar.
 	
 	\item If $\alpha_1\geq 3$ and $\alpha_2\geq 2$, then $G$ has a subgroups $H_1$, $H_2$, $H_3$, $N_1$, $N_2$, $N_3$ of orders $p_1$, $p_2$, $p_1p_2$, $p_1^2p_2$, $p_1^2p_2^2$, $p_1p_2^2$, respectively and so $K_{3,3}$ is a subgraph of $\mathcal I(G)$ with bipartition $X:=\{H_1$, $H_2$, $H_3\}$ and 
 	$Y:=\{N_1$, $N_2$, $N_3\}$. This implies that $\mathcal I(G)$ is non-planar.
 	
 	\item If $\alpha_1 = 4$ and $\alpha_2 =1$, then $G$ has  subgroups of order $p_1$, $p_1^2$, $p_1^3$, $p_1^4$, $p_2$, $p_1p_2$, $p_1^2p_2$, $p_1^3p_2$, let them be $H_i$, $i=1,2,\ldots , 8$, respectively and these are the only proper subgroups of $G$. Here $\mathcal I(G)$ is planar and a corresponding plane embedding is shown in Figure~\ref{incl:planar}(a).

 	\begin{figure}
 	    \centering
 	    \begin{subfigure}[b]{0.3\textwidth}
 	        \includegraphics[width=\textwidth]{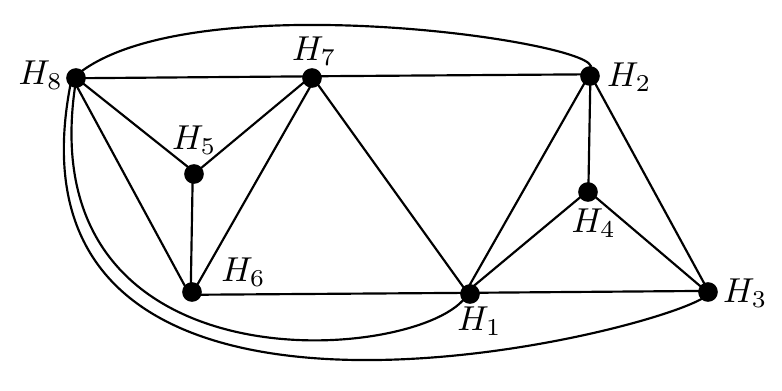}
 	        \caption{}
 	        \label{incl fig7}
 	    \end{subfigure}
 	            \begin{subfigure}[b]{0.2\textwidth}
 	        \includegraphics[width=\textwidth]{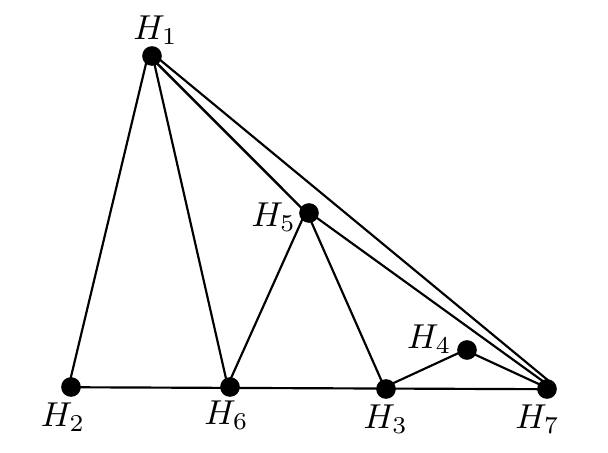}
 	        \caption{}
 	        \label{incl fig8}
 	    \end{subfigure}
 	     	       \begin{subfigure}[b]{0.3\textwidth}
 	        \includegraphics[width=\textwidth]{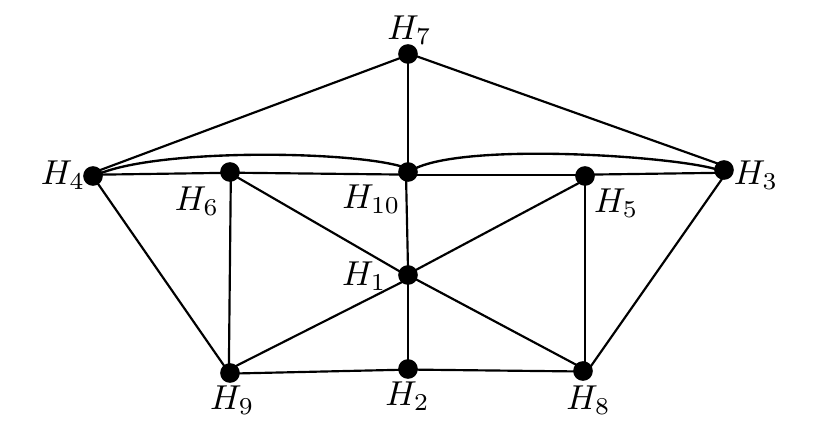}
 	        \caption{}
 	        \label{incl fig9}
 	    \end{subfigure}
 	    
 	     	    \begin{subfigure}[b]{0.3\textwidth}
 	            \includegraphics[width=\textwidth]{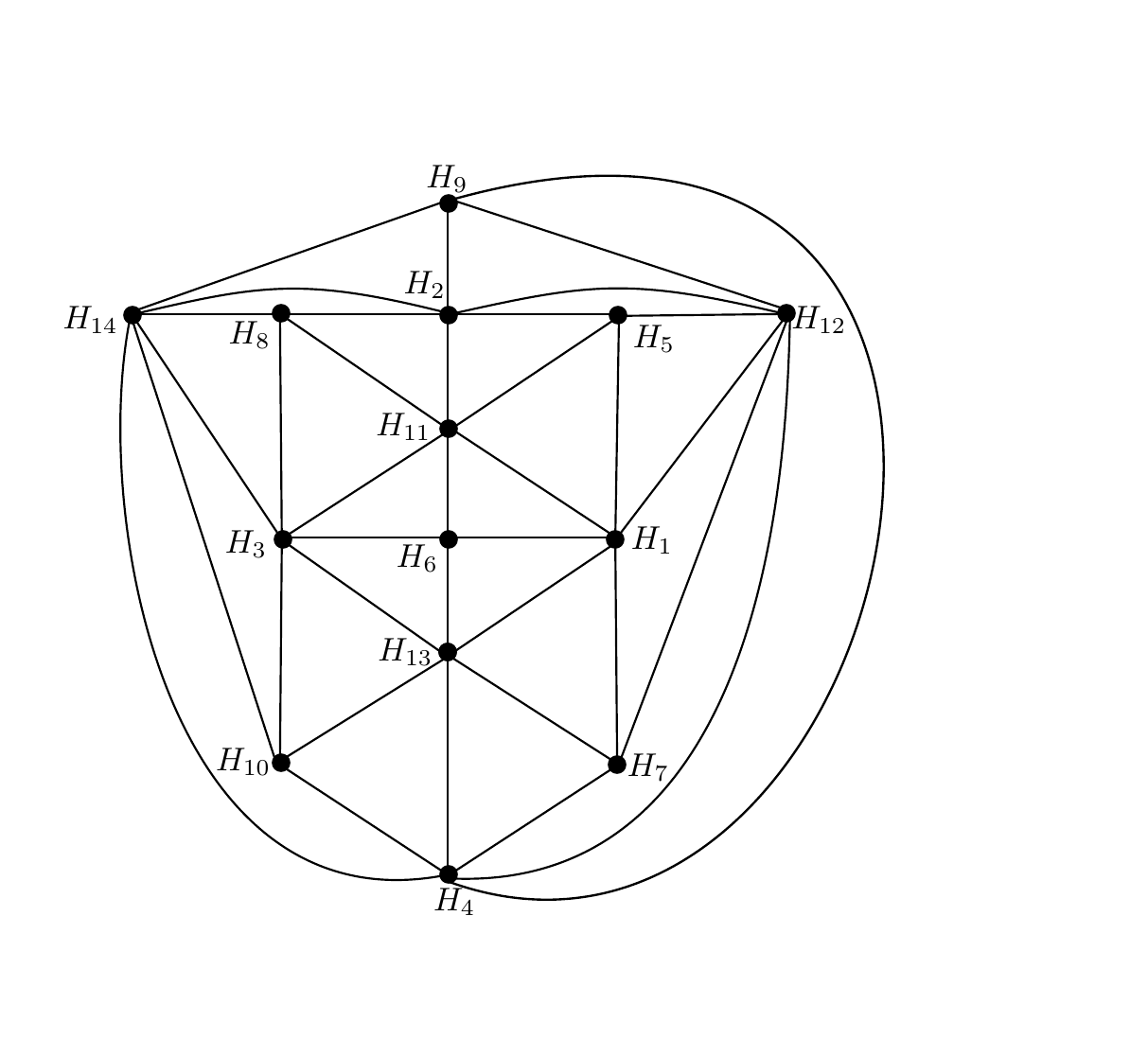}
 	            \caption{}
 	            \label{incl fig10}
 	        \end{subfigure}
 	        \begin{subfigure}[b]{0.3\textwidth}
 	                \includegraphics[width=\textwidth]{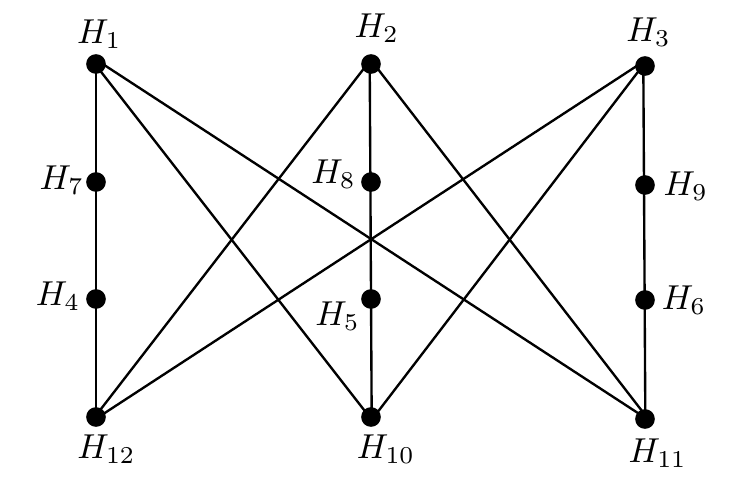}
 	                \caption{}
 	                \label{incl fig11}
 	            \end{subfigure}
\caption{(a) $\mathcal I(\mathbb Z_{p_1^4p_2})$, (b) $\mathcal I(\mathbb Z_{p_1^2p_2^2})$, (c) $\mathcal I(\mathbb Z_{p_1^2p_2p_3})$, (d) $\mathcal I(\mathbb Z_{p_1p_2p_3p_4})$, (e) a subdivision of $K_{3,3}$ in $\mathcal I(\mathbb Z_p\times \mathbb Z_p\times \mathbb Z_p)$}\label{incl:planar}
\end{figure}
 	

 	
 	\item If $\alpha_1 \leq 3$ and $\alpha_2 =1$, then $\mathcal I(G)$ is a subgraph of $\mathcal I(\mathbb Z_{p_1^4p_2})$ and so $\mathcal I(G)$ is planar.
 	
 	\item If $\alpha_1 =2$ and $\alpha_2 =2$, then $G$ has subgroups of order $p_1$, $p_1^2$, $p_2$, $p_2^2$, $p_1p_2$, $p_1^2p_2$, $p_1p_2^2$, let them be $H_i$, $i=1,2,\ldots , 7$ respectively and these are the only proper subgroups of $G$. Here $\mathcal I(G)$ is planar and a corresponding plane embedding is  shown in Figure~\ref{incl:planar}(b).
 \end{enumerate}
 
 \noindent \textbf{Subcase 1b:} $k=3$. 
 \begin{enumerate}[\normalfont (i)]
 	\item  If $\alpha_1 \geq  3$, $\alpha_2 =1=\alpha_3$,  then let $H_1$, $H_2$, $H_3$, $N_1$, $N_2$, $N_3$ be subgroups of $G$ of orders $p_1$, $p_2$, $p_1p_2$, $p_1p_2p_3$, 
 	$p_1^2p_2$, $p_1^2p_2p_3$, respectively. Then $K_{3,3}$ is a subgraph of $\mathcal I(G)$ with bipartition $X:=\{H_1$, $H_2$, $H_3\}$ and 
 	$Y:=\{N_1$, $N_2$, $N_3\}$ and so $\mathcal I(G)$ is non-planar.
 	
 	\item If $\alpha_1$, $\alpha_2\geq 2$, $\alpha_3\geq 1$, then we can use a similar argument as in (i) above to show $\mathcal I(G)$ is non-planar.
 	
 	\item $\alpha_1 =2 $, $\alpha_2 =1=\alpha_3$, then $G$ has subgroups of order $p_1$, $p_1^2$, $p_2$, $p_3$, $p_1p_2$, $p_1p_3$,  $p_1^2p_2$, $p_2p_3$, $p_1^2p_3$, $p_1p_2p_3$, let them be $H_i$, $i=1,\ldots , 10$ respectively and these are the only proper subgroups of $G$. Here $\mathcal I(G)$ is planar and a corresponding plane embedding is shown in Figure~\ref{incl:planar}(c).
 	
 	\item If $\alpha_1 =\alpha_2 =\alpha_3 =1$, then $\mathcal I(G)$ is a subgraph of $\mathcal I(\mathbb Z_{p^2qr})$ and so $\mathcal I(G)$ is planar.
 	
 \end{enumerate}
 
 \noindent \textbf{Subcase 1c:} $k=4$. 
 \begin{enumerate}[\normalfont (i)]
 
	\item If  $\alpha_1\geq 2$, $\alpha_2$, $\alpha_3$, $\alpha_4\geq 1$, then let $H_1$, $H_2$, $H_3$, $N_1$, $N_2$, $N_3$ be subgroups of $G$ of orders $p_1$, $p_2$, $p_1p_2$, $p_1p_2p_3$, 
 $p_1p_2p_4$, $p_1^{\alpha_1}p_2p_3$, respectively and so $K_{3,3}$ is a subgraph of $\mathcal I(G)$ with bipartition $X:=\{H_1$, $H_2$, $H_3\}$ and 
 $Y:=\{N_1$, $N_2$, $N_3\}$. This implies that $\mathcal I(G)$ is non-planar.
 
 	\item If $\alpha_1=\alpha_2=\alpha_3= \alpha_4 = 1$, then $G$ has subgroups of order $p_1$, $p_2$, $p_3$, $p_4$, $p_1p_2$, $p_1p_3$, $p_1p_4$, $p_2p_3$, $p_2p_4$, $p_3p_4$, $p_1p_2p_3$, $p_1p_2p_4$, $p_1p_3p_4$, $p_2p_3p_4$, let them be $H_i$, $i=1,2, \ldots , 14$ respectively and these are the only proper subgroups of $G$. Here $\mathcal I(G)$ is planar and a corresponding plane embedding is shown in Figure~\ref{incl:planar}(d).
   \end{enumerate}
 \noindent \textbf{Subcase 1d:} $k\geq 5$.  Let  $H_1$, $H_2$, $H_3$, $N_1$, $N_2$, $N_3$ be subgroups of $G$ of orders $p_1$, $p_2$, $p_1p_2$, $p_1p_2p_3$, 
 $p_1p_2p_4$, $p_1p_2p_5$, respectively and so $K_{3,3}$ is a subgraph of $\mathcal I(G)$ with bipartition $X:=\{H_1$, $H_2$, $H_3\}$ and 
 $Y:=\{N_1$, $N_2$, $N_3\}$. This implies that $\mathcal I(G)$ is non-planar.
%
%
%

\noindent \textbf{Case 2:} Let $G$ be non-cyclic.

 \begin{enumerate}[\normalfont (i)]
\item If $G\cong \mathbb Z_p\times \mathbb Z_p$, then by Corollary~\ref{inclusion graphs 3}(1), $\mathcal I(G)$ is planar.

\item If $G\cong \mathbb Z_{p^2}\times \mathbb Z_p$, then  by  Figure~\ref{incl fig}(a), $\mathcal I(G)$ is  planar.

\item If $G\cong \mathbb Z_{pq}\times \mathbb Z_p$, then by Figure~\ref{incl fig}(c), $\mathcal I(G)$ is planar.

\item If $G\cong \mathbb Z_{p^2q}\times \mathbb Z_p$, then $\mathbb Z_{p^2}\times \mathbb Z_p$ is a subgroup of $G$ and by Figure~\ref{incl fig}(a), 
$G$ has a unique subgroup of order $p$, say $H$; let $H_1$, $H_2$, $H_3$ be subgroups of $\mathbb Z_{p^2}\times \mathbb Z_p$ of order 
$p^2$; let $N$ be a subgroup of $G$ of order $q$. 
Then $K_{3,3}$ is a subgraph of $\mathcal I(G)$ with bipartition $X:=\{H_1N$, $H_2N$, $H_3N\}$, $Y:=\{N$, $H$ $HN\}$ and so $\mathcal I(G)$ is non-planar.

 \item If $G\cong \mathbb Z_{p^2}\times \mathbb Z_{p^2}=\langle a,b~|~a^{p^2}=b^{p^2}, ab=ba\rangle$, then  $K_{3,3}$ is a subgraph of $\mathcal I(G)$ with bipartition $X:=\{\langle a, b^p\rangle$ $\langle a^p,b\rangle$, $\langle a^p,b^p\rangle\}$, $Y:=\{\langle a^p\rangle$, $\langle b^p\rangle$, $\langle a^pb^p\rangle\}$ and so $\mathcal I(G)$ is non-planar.

\item If $G\cong \mathbb Z_p\times \mathbb Z_p\times \mathbb Z_p$, then $H_1:=\langle (1,0,0)\rangle$, $H_2:=\langle (0,1,0)\rangle$, $H_3:=\langle (0,0,1)\rangle$, $H_4:=\langle (0,1,1)\rangle$, $H_5:=\langle (1,0,1)\rangle$, $H_6:=\langle (1,1,0)\rangle$, $H_7:=\langle (1,0,0), (0,1,1)\rangle$, $H_8:=\langle (0,1,0), (1,0,1)\rangle$, $H_9:=\langle (0,0,1), (1,1,0)\rangle$, $H_{10}:=\langle (1,0,0), (0,0,1)\rangle$, $H_{11}:=\langle (1,0,0),$ $(0,1,0)\rangle$, $H_{12}:=\langle (0,1,0), (0,0,1)\rangle$ are subgroups of $G$. Here $\mathcal I(G)$ has a subdivision of $K_{3,3}$ as a 
subgraph, which is shown in Figure~\ref{incl:planar}(e) and so  $\mathcal I(G)$ is non-planar.

 \item If $G\cong \mathbb Z_{p_1^{\alpha_1}}\times \mathbb Z_{p_2^{\alpha_2}}\times \cdots \mathbb \times Z_{p_k^{\alpha_k}}$, where $k \geq 3$, $p_i$'s are primes and at least two of them are equal (since $G$ is non-cyclic, all the primes cannot be distinct here), $\alpha_i \geq 1$ are integers,
 $i=1, \ldots,k$. Then,    $G$  has one of $\mathbb Z_{p_i^2p_j}\times \mathbb Z_{p_i}$, $\mathbb Z_{p_i^2}\times \mathbb Z_{p_i^2}$ or $\mathbb Z_{p_i}\times \mathbb Z_{p_i}\times \mathbb Z_{p_i}$ as a subgroup for some $i$ and $j$, $i\neq j$. So by the above arguments, it follows 
that $\mathcal I(G)$ is non-planar.
 \end{enumerate}
Proof follows by putting together the above cases.
\end{proof}


\begin{thm}\label{inclusion graphs 11}
 If $G$ is a finite abelian group, then $diam(\mathcal I(G)) \in \{1, 2, 3, 4, \infty\}$.
\end{thm}
\begin{proof}
If $\mathcal I(G)$ is disconnected, then $diam(\mathcal I(G))=\infty$. Now we assume that $\mathcal I(G)$ is connected.	
If $G\cong \mathbb Z_{p^\alpha}$, where $p$ is a prime and $\alpha >1$, then by Theorem~\ref{inclusion graph 112}, it follows that $diam(\mathcal I(G))=1$. Now we assume that $G \ncong\mathbb Z_{p^\alpha}$. Since $\mathcal I(G)$ is connected, so the order of not every proper subgroup of $G$ is a prime. Let $H$, $K$ be two proper subgroups of $G$.  If $|H|$, $|K|$ are prime, then $HK$ is a proper subgroup of $G$ and so $H-HK-K$ is a path from $H$ to $K$. If $|H|$ is a prime, $|K|$ is not a prime, then  $H-HK'-K'-K$ is a path from $H$ to $K$, where $K'$ is a proper subgroup of $K$. If $|H|$, $|K|$ are not primes and $|H\cap K|=1$, then 
$H-H'-H'K'-K'-K$ is a path from $H$ to $K$, where $H'$, $K'$ are proper subgroups of $H$, $K$ respectively. If $|H|$, $|K|$ are not primes and $|H\cap K|>1$, then $H-H\cap K-K$ is a path from $H$ to $K$. Thus we have shown that $diam(\mathcal I(G))\leq 4$.
It is easy to see that $diam(\mathcal I(\mathbb Z_{p^2q^2}))=2$,   $diam(\mathcal I(\mathbb Z_{p^2}\times \mathbb Z_p))=3$,  $diam(\mathcal I(\mathbb Z_{pq}\times \mathbb Z_{pq}))=4$. So this shows that $diam(\mathcal I(G))$ takes all the values in $\{1,2,3,4, \infty\}$ and the proof is complete.
\end{proof}

Our next aim is to prove the following result, which describes the girth of the intersection graph of subgroups of finite groups and classifies all the finite groups whose inclusion graph of subgroups is $K_{1,3}$-free.

\begin{thm}\label{inclusion graphs 7}
	Let $G$ be a finite group and $p$, $q$ be distinct primes. Then  
	\begin{enumerate}[{\normalfont (1)}]
		\item $girth(\mathcal I(G))\in\{3, 6, \infty\}$;
		\item $\mathcal I(G)$ is $K_{1,3}$-free if and only if $G$ is one of $\mathbb Z_{p^\alpha}(\alpha= 2,3,4)$, $\mathbb Z_{p^\alpha q}(\alpha=1,2)$, 
		$\mathbb Z_{pqr}$, $\mathbb Z_p\times \mathbb Z_p$ or $\mathbb Z_q\rtimes \mathbb Z_p$;
	\end{enumerate}
\end{thm}
To prove the above theorem, we start with the following result.

\begin{pro}\label{inclusion p1}
	Let $G$ be a group of order $p^{\alpha}$, where $p$ is a prime and $\alpha \geq 2$. Then
	\begin{enumerate}[\normalfont (1)]
		\item $girth(\mathcal I(G))\in \{3, 6, \infty \}$;
		\item $\mathcal I(G)$ is $K_{1,3}$-free if and only if $G$ is either $\mathbb Z_{p^{\alpha}}(\alpha=2, 3, 4)$ or $\mathbb Z_{p}\times \mathbb Z_{p}$.

	\end{enumerate}
\end{pro}
\begin{proof} Proof is divided in to two cases.
	
	\noindent \textbf{Case 1:}	 Let $\alpha \geq 4$. Then $G$ has a chain of subgroups of length at least four and so $\mathcal I(G)$ contains $C_3$ as a subgraph. Hence $girth(\mathcal I(G))=3$.  If  $G\cong \mathbb Z_{p^{\alpha}}$, then by Theorem~\ref{inclusion graph 112},  $\mathcal I(G)$ is  $K_{1,3}$-free only when $\alpha = 2,3,4$. If $G\ncong \mathbb Z_{p^\alpha}$, then $G$ has at least two subgroups of order $p^{\alpha-1}$, let them be $H_1$, $H_2$. Also since $|H_1\cap H_2|=p^{\alpha_1-2}$, $H_1\cap H_2$ has a subgroup of order $p_1$, let it be $H_3$. It follows that $\mathcal I(G)$ contains $K_{1,3}$ as a subgraph with bipartition $X:=\{H_3\}$ and $Y:=\{H_1$, $H_2$, $H_1\cap H_2\}$. 
	
	\noindent \textbf{Case 2:} If $\alpha\leq 3$, then $G$ is isomorphic to one of $\mathbb Z_{p^{\alpha}}$, $\mathbb Z_p\times \mathbb Z_p$, $\mathbb Z_{p^2}\times \mathbb Z_{p}$, $\mathbb Z_{p}\times \mathbb Z_{p}\times \mathbb Z_{p}$, $Q_8$, $M_8$, $M_{p^3} (p >2)$ or $(\mathbb Z_{p}\times \mathbb Z_{p})\rtimes \mathbb Z_{p}$.
	\begin{enumerate}[\normalfont (i)]
%
		
		
		\item If $G\cong \mathbb Z_{p}\times \mathbb Z_{p}\times \mathbb Z_{p}:=\langle a, b, c~|~a^{p}=b^{p}=c^{p}=1, ab=ba, ac=ca, bc=cb\rangle$,  
		then $\mathcal I(G)$ contains $K_{1,3}$ as a subgraph with bipartition $X:=\{\langle a,b\rangle\}$ and $Y:=\{\langle a\rangle$, $\langle b\rangle$, $\langle ab\rangle\}$. By Case 2 in the proof of Theorem~\ref{incl 121}, we proved that $\mathcal I(G)$ contains $C_6$ as a subgraph. Also by Corollary~\ref{inclusion graphs 3}(2), $\mathcal I(G)$ is bipartite and so $girth(\mathcal I(G))=4$ or 6. Suppose that $\mathcal I(G_2)$ contains $C_4$ as a subgraph. Let it be $H_1 -H_2- H_3- H_4-H_1$.  Then either $H_1$, $H_3 \subset H_2$, $H_4$ or $H_2$, $H_4 \subset H_1$, $H_3$. Without loss of generality, we may assume that $H_1$, $H_3\subset H_2$, $H_4$. Then we must have $|H_1|=p=|H_3|$, $|H_2|=p^2=|H_4|$ and this implies that $\langle H_1, H_2\rangle=H_2=H_4$, which is not possible. so $girth(\mathcal I(G))= 6$.
		
%
%
		
		\item 	If $G\cong \mathbb (Z_{p}\times \mathbb Z_{p})\rtimes \mathbb Z_{p}:=\langle a, b, c~|~a^{p}=b^{p}=c^{p}=1, ab=ba, ac=ca, cbc^{-1}=ab\rangle$, 
		then $\mathcal I(G)$ contains $K_{1,3}$ as a subgraph with bipartition $X:=\{\langle a,b\rangle\}$ and $Y:=\{\langle a\rangle$, $\langle b\rangle$, $\langle ab\rangle\}$. In Case 2 in the proof of Theorem~\ref{incl 121}, we already proved that $\mathcal I(G)$ contains $C_6$ as a subgraph. Also by Corollary~\ref{inclusion graphs 3}(2), $\mathcal I(G)$ is bipartite and so $girth(\mathcal I(G))=4$ or 6. As in part (i) of this Case, one can easily show that $girth(\mathcal I(G))=6$. 
		
		\item By Theorem~\ref{inclusion graph 112}, Corollary~\ref{inclusion graphs 3}(1), Figures~\ref{incl fig}(a),~\ref{incl fig}(b) and~\eqref{incle2}, we see that the girth of the inclusion graph of subgroups of all the remaining groups is infinity and except $\mathbb Z_{p^2}\times \mathbb Z_{p}$, they contains $K_{1,3}$ as a subgraph. 
		\end{enumerate} 		
	Proof follows from the above two cases.
\end{proof}

	\begin{pro}\label{inclusion p2}
		Let $G$ be a group of order $p^{\alpha}q^{\beta}$, where $p$, $q$ are distinct primes and $\alpha, \beta \geq 1$. Then
		\begin{enumerate}[\normalfont (1)]
			\item $girth(\mathcal I(G))\in \{3, 6, \infty \}$;
			\item $\mathcal I(G)$ is $K_{1,3}$-free if and only if $G$ is either  $\mathbb Z_{pq}$ or $\mathbb Z_{p^2q}$.
		\end{enumerate}
	\end{pro}
	\begin{proof} Proof is divided in to several cases.
		
	\noindent \textbf{Case 1:}	 If $\alpha =1=\beta$, then $G\cong \mathbb Z_{pq}$ or $\mathbb Z_{p}\rtimes \mathbb Z_{q}$. By Corollary~ \ref{inclusion graphs 3}(1), $\mathcal I(G)$ is totally disconnected and so $\mathcal I(G)$ is $K_{1,3}$-free and $girth(\mathcal I(G))=\infty$.
	
	\noindent \textbf{Case 2:} If $\alpha=2$, $\beta =1$, then we need to consider the following subcases.
	
	\noindent \textbf{Subcase 2a:} Let $G$ be abelian. If $G\cong \mathbb Z_{p^2q}$, then by~\eqref{e4}, $\mathcal I(G)$ is $K_{1,3}$-free and $girth(\mathcal I(G))=\infty$. If $G\cong \mathbb Z_{pq}\times \mathbb Z_{p}$, then by Figure~\ref{incl fig}(c), $\mathcal I(G)$ has $K_{1,3}$ as a subgraph and $girth(\mathcal I(G))=6$. 
	
	\noindent \textbf{Subcase 2b:}
	 Let $G$ be non-abelian. We proceed with the groups considered in Subcase 3b in the proof of Theorem~\ref{incl 121}.
	\begin{enumerate}[\normalfont (i)]
		
		\item 	Since $\mathbb Z_p\times \mathbb Z_p$ 
		is a subgroup of $G_2$, so $\mathbb Z_p\times \mathbb Z_p$ together with its proper subgroups forms $K_{1,3}$ as a subgraph of $\mathcal I(G_2)$.
		Moreover, we already proved that $\mathcal I(G_2)$ contains $C_6$ as a subgraph.  By Corollary~\ref{inclusion graphs 3}(2), $\mathcal I(G_2)$ is bipartite and so $girth(\mathcal I(G_2))=4$ or 6. Suppose that $\mathcal I(G_2)$ contains $C_4$ as a subgraph. Let it be $H_1 -H_2- H_3- H_4-H_1$.  Then either $H_1$, $H_3 \subset H_2$, $H_4$ or $H_2$, $H_4 \subset H_1$, $H_3$. Without loss of generality we may assume that $H_1$, $H_3\subset H_2$, $H_4$.
		\begin{itemize}
			\item If $|H_1|=p_1=|H_3|$, $|H_2|=p_1^2=|H_4|$, then $\langle H_1, H_2\rangle=H_2=H_4$, which is not possible.
			\item If $|H_1|=p_1=|H_3|$, $|H_2|=p_1^2$, $|H_4|=p_1p_2$, then $H_3\nsubseteq H_4$ if $H_1\subset H_4$; $H_1\nsubseteq H_4$,  if $H_3\subset H_4$, which are not possible.
			\item If $|H_1|=p_1=|H_3|$, $|H_2|=p_1p_2=|H_4|$, then also we get a contradiction as above.
			\item If $|H_1|=p_1$, $|H_3|=p_2$, $|H_2|=p_1p_2=|H_4|$, then  $\langle H_1, H_2\rangle=H_2=H_4$, which is not possible. 
		\end{itemize}
		
		Hence $girth(\mathcal I(G_2))=6$.
		
%
		
		\item We already proved that $\mathcal I(G_{5(t)})$ contains $C_6$ as a subgraph. Also $\mathbb Z_p\times \mathbb Z_p$ is a subgroup of $G_{5(t)}$, so $\mathbb Z_p\times \mathbb Z_p$ together with its proper subgroups forms $K_{1,3}$ as a subgraph of $\mathcal I(G_{5(t)})$. By Corollary~\ref{inclusion graphs 3}(2), $\mathcal I(G_{5(t)})$ is bipartite and so $girth(\mathcal I(G))=4$ or 6. As in part (i) of this Subcase, one can easily show that $girth(\mathcal I(G))=6$.



	\item	By Figures~\ref{incl fig}(d),~\ref{incl fig}(e),~\ref{incl fig}(f),~$\eqref{e17}$,~\eqref{e8}, we see that the girth of the inclusion graph of subgroups of all the remaining groups is infinity and  they contains $K_{1,3}$ as a subgraph. 
	\end{enumerate}

%
	
	\noindent \textbf{Case 3:} Let $\alpha \geq 3$ and $\beta =1$. Then $G$ has a chain of subgroups of length at least four and so $girth(\mathcal I(G))=3$. If $G$ is cyclic, then let $H_1$, $H_2$, $H_3$, $H_4$ be the subgroups of $G$ of orders $p$, $q$, $p^2$, $p^2q$ 
	respectively. Then $\mathcal I(G)$ contains $K_{1,3}$ as a subgraph with bipartition $X:=\{H_4\}$ and $Y:=\{H_1, H_2, H_3\}$. If $G$ is non-cyclic abelian, then $\mathbb Z_{pq}\times \mathbb Z_{p}$ 
	is a proper subgroup of $G$, so by Subcase 2a, $\mathcal I(G)$ contains $K_{1,3}$ as a subgraph. 
	
Now we assume that $G$ is non-abelian. Let $P$ denote a Sylow $p$-subgroups of $G$. We shall prove that $\mathcal I(G)$ contains $K_{1,3}$ as a subgraph. First 
	let $\alpha_1=3$. If $p>q$, then $n_{p}(G)=1$, by Sylow's theorem and our group $G \cong P \rtimes \mathbb Z_{q}$. Suppose $\mathcal I(P)$ is contains $K_{1,3}$,  
	then $\mathcal I(G)$ also contains $K_{1,3}$, so it is enough to consider the case when $\mathcal I(G)$ is $K_{1,3}$-free. 
	By Proposition~\ref{inclusion p1},
	we must have $P \cong \mathbb Z_{p^3}$. 
	Then $G \cong \mathbb Z_{p^3} \rtimes \mathbb Z_{q}:=\langle a, b~|~a^{p^3}=b^q=1, bab^{-1}=a^i, ord_{p^3}(i)=q\rangle$ and $\mathcal I(G)$ contains $K_{1,3}$ as a subgraph with bipartition 
	$X:=\{\langle a^p, b \rangle\}$ and $Y:=\{\langle a^p \rangle, \langle a^{p^2} \rangle, \langle b \rangle\}$.
	
	Now, let us consider the case $p< q$ and $(p,q)\neq (2$, $3)$. Here $n_{q}(G) = p$ is not possible. If $n_{q}(G)= p^{2}$, then $q| (p+1)(p-1)$ which implies that
	$q| (p+1)$ or $q| (p-1)$. But this is impossible, since $q> p > 2$. If $n_{q}(G)= p^3$, then there are $p^{3}(q-1)$ elements of order $q$.
	However, this only leaves $p^{3}q-p^{3}(q-1)= p^3$ elements, and the Sylow p-subgroup must be normal, a case we already considered. Therefore,
	the only remaining possibility is that $G\cong \mathbb Z_q \rtimes P$. Suppose $\mathcal I(P)$ contains $K_{1,3}$,  
	then so is $\mathcal I(G)$. So it is enough to consider the case when $\mathcal I(G)$ is $K_{1,3}$-free. 
	By Proposition~\ref{inclusion p1}, 
	we must have $P \cong \mathbb Z_{p^3}$. Then $G\cong \mathbb Z_q\rtimes \mathbb Z_{p^3}=\langle a,b~|~a^q=b^{p^3}=1, bab^{-1}=a^i, \mbox{ord}_q(i)=p\rangle$, $p~|~(q-1)$ and $\mathcal I(G)$ contains $K_{1,3}$ as a subgraph with 
	bipartition $X:=\{\langle a, b^p \rangle\}$ and $Y:=\{\langle b^{p^2} \rangle, \langle b^p \rangle, \langle a \rangle\}$.
	
	If $(p,q)= (2, 3)$, then \cite[p.160]{burn} states that the only group of order 24 that is not a semi-direct product is $S_4$. Also $A_4$ is a subgroup of $S_4$. 
	So by \eqref{e8}, $\mathcal I(A_4)$ contains $K_{1,3}$ as a subgraph. Thus the result is true when $\alpha =3$.
	
 Next, let $\alpha_1>3$. Then $P$ has  a chain of subgroups of length  at least 4 and so $\mathcal I(G)$ has $K_4$ as a subgraph. This implies that $\mathcal I(G)$ has $K_{1,3}$ as a subgraph. 
	
	\noindent \textbf{Case 4:} Let $\alpha_1$, $\beta\geq 2$. Since $G$ is solvable, so it has a normal subgroup with prime index, say $p_2$, let it be $H$.  Let $H_1$, $H_2$, $H_3$ be subgroups of $H$  of order $p$, $p^2$, $q$ respectively, such that $ H_1 \subset H_2$.  Then $H, H_1, H_2$ forms $C_3$ as a subgraph of $\mathcal I(G)$ and so $girth(\mathcal I(G))=3$. Also, $\mathcal I(G)$ contains $K_{1,3}$ as a subgraph with bipartition $X:=\{H\}$ and $Y:=\{H_1$, $H_2$, $H_3\}$. 
\end{proof}

\begin{pro}\label{incl 1}
		Let $G$ be a solvable group whose order has at least three distinct prime factors. Then
		\begin{enumerate}[\normalfont (1)]
			\item $girth(\mathcal I(G))\in \{3, 6\}$;
			\item $\mathcal I(G)$ is $K_{1,3}$-free if and only if $G \cong \mathbb Z_{pqr}$, where $p$, $q$, $r$ are distinct primes.
		\end{enumerate}
\end{pro}
\begin{proof} Let $|G|=p_1^{\alpha_1}p_2^{\alpha_2}\cdots p_k^{\alpha_k}$, where $p_i$'s are distinct primes, $k \geq 3$, $\alpha_i \geq 1$.
	
	\noindent \textbf{Case 1:} Let  $k=3$. If $\alpha_1=\alpha_2=\alpha_3=1$, then in Case 4 in the proof of Theorem~\ref{incl 121}, we already proved that $\mathcal I(G)$ contains $C_6$ as a subgraph. Also by Corollary~\ref{inclusion graphs 3}(2), $\mathcal I(G)$ is bipartite and so $girth(\mathcal I(G))=4$ or 6. 
	By using a similar argument as in (i) of Subcase 2b in the proof of  Proposition~\ref{inclusion p2}, one can easily see that $girth(\mathcal I(G))=6$. If $G\cong \mathbb Z_{pqr}$, then by Theorem~\ref{incl 121}(1), $\mathcal I(G)$ is $K_{1,3}$-free. If $G$ is non-abelian, then  $G$ has a non-cyclic subgroup of composite order, let it be $H$. It follows that $H$ together with its proper subgroups forms $K_{1,3}$ as a subgraph of $\mathcal I(G)$.
Now let $\alpha_1\geq 2$, $\alpha_2$, $\alpha_3\geq 1$. Since $G$ is solvable, so it has a subgroup of order $p_1^{\alpha_1}p_2$, let it be $H$.	 
	 	  Let $H_1$, $H_2$, $H_3$ be subgroups of $H$ of order $p_1$, $p_1^{\alpha_1}$, $p_2$ respectively with $H_1\subset H_2$. It follows that $girth(\mathcal I(G))=3$ and $\mathcal I(G)$ contains $K_{1,3}$ as a subgraph with bipartition $X:=\{H\}$ and $Y:=\{H_1$, $H_2$, $H_3\}$.

	\noindent \textbf{Case 2:} $k\geq 4$. Since $G$ is solvable, so it has a subgroup of order $p_1^{\alpha_1}p_2^{\alpha_2}p_3^{\alpha_3}$, let it be $H$. Let $H_1$, $H_2$, $H_3$ be subgroups of $H$ of order $p_1$, $p_1^{\alpha_1}p_2$, $p_2$ respectively with $H_1\subset H_2$.  It follows that $girth(\mathcal I(G))=3$ and $\mathcal I(G)$ contains $K_{1,3}$ as a subgraph with bipartition $X:=\{H\}$ and $Y:=\{H_1$, $H_2$, $H_3\}$. 
	
	Proof follows from the above two cases.	
%
%
\end{proof}




It is well known that any non-solvable group has a simple group as a sub-quotient and every simple group has a minimal simple group as a sub-quotient. 
So if we can show that the inclusion graph of subgroups of a minimal simple group contains a graph $X$ as a subgraph, then by Theorem~\ref{inclusion graph 13}, the inclusion graph of subgroups of a non-solvable group also contains $X$.

Recall that $SL_m(n)$ is the group of $m \times m$ matrices having determinant 1, whose entries are lie in a field with $n$ elements and that 
$L_m(n)=SL_m(n)/ H$, where $H=\{kI| k^m=1 \}$. For any prime $q > 3$, the Suzuki group is denoted by $Sz(2^q)$. For any integer $n\geq 3$, the  dihedral group of order $2n$ is given by $D_{2n}=\langle a, b~|~a^n=b^2=1, ab=ba^{-1}\rangle$. 
Note that $K_{1,3}$ is a subgraph of $\mathcal I(D_{4n})$ with bipartition $X:=\{\langle a^2, b\rangle\}$ and 
$Y:=\{\langle a^2\rangle, \langle a^2b\rangle, \langle b\rangle \}$. 
%
%

\begin{lemma}\label{inclusion graph l1}
	$\mathcal I(D_{4n})$ contains $K_{1,3}$ as a subgraph, for $n \geq 3$.
\end{lemma}

\begin{pro}\label{inclusion graph t11}
	If $G$ is a finite non-solvable group, then $\mathcal I(G)$ has $K_{1,3}$ as a subgraph and $girth(\mathcal I(G))=3$.
\end{pro}
\begin{proof}
	If we show that $\mathcal I(G)$ contains $K_2+\overline{K}_3$ as a subgraph, then it follows that $\mathcal I(G)$ contains $K_{1,3}$ as a subgraph and $girth(\mathcal I(G))=3$. As mentioned above, to prove this, it is enough to show that the inclusion graph of subgroups of  minimal simple groups contains $K_2+\overline{K}_3$ as a subgraph. We use the J. G. Thompson's classification of minimal simple groups given in~\cite{thomp} and check this for this list of groups.  We will denote the image of a matrix $A$ in $L_m(n)$ by $\overline{A}$.
	
	\noindent \textbf{Case 1:}  $G \cong L_2(q^p)$. 
	If $p=2$, then the only non-solvable group is $ L_2(4)$ and $ L_2(4) \cong A_5$. Also $A_4$ is a subgroup of $A_5$, and so by~\eqref{e8}, 
	$A_4$ together with its proper subgroups 
	forms $K_2+\overline{K}_3$ as a subgraph of $\mathcal I(G)$.
	If $p> 2$, then $ L_2(q^p)$ contains a subgroup isomorphic to $(\mathbb Z_q)^p$,
	namely the subgroup of matrices of the form $\overline{
		\left( \begin{smallmatrix}
		1 & a\\
		0 & 1
		\end{smallmatrix}\right)
	}$ with $a \in \mathbb F_{q^p}$. By Proposition~\ref{inclusion p1}, $\mathcal I((\mathbb Z_q)^p )$, $p >2$ 
	contains $K_{1,3}$ as a subgraph and so $(\mathbb Z_q)^p$ together with proper its subgroups forms 
	$K_2+\overline{K}_3$ as a subgraph of $\mathcal I(G)$. 
		
\noindent \textbf{Case 2:}: $G \cong  L_3(3)$. Note that $L_3(3) \cong SL_3(3)$.
	Let us consider the subgroup consisting of matrices of the form
	$\left(\begin{smallmatrix}
	1 & a & b\\
	0 & 1 & c\\
	0 & 0 & 1
	\end{smallmatrix}\right)$
	with $a, b, c \in \mathbb F_3$. This subgroup is isomorphic to the group 
	$(\mathbb Z_p \times \mathbb Z_p) \rtimes \mathbb Z_p$ with $p=3$. By Proposition~\ref{inclusion p1}, $\mathcal I((\mathbb Z_p
	\times \mathbb Z_p) \rtimes \mathbb Z_p)$ contains $K_{1,3}$ as a subgraph. So
	$(\mathbb Z_p \times \mathbb Z_p) \rtimes \mathbb Z_p$ together with its proper subgroups forms $K_2+\overline{K}_3$ as a subgraph of $\mathcal I(G)$.
	
	\noindent \textbf{Case 3:} $G \cong L_2(p)$. Note that here $H=\{\pm I\}$. We have to consider two subcases:
	
	\noindent \textbf{Subcase 3a:} $p \equiv 1  ~(\text{mod}~ 4)$. Then  $L_2(p)$ 
	has a  subgroup isomorphic to $D_{p-1}$\cite[p.~222]{boh-reid}. So, by Lemma~\ref{inclusion graph l1}, $D_{p-1}$ together with its proper subgroups forms $K_2+\overline{K}_3$ as a subgraph of $\mathcal I(G)$, when $p>5$. 
	If $p=5$, then $L_2(5) \cong A_5 \cong L_2(4)$, which we already dealt.
	
\noindent \textbf{Subcase 3b:} $p \equiv 3 ~(\text{mod}~ 4)$.  $L_2(p)$ 
	has  a subgroup isomorphic to $D_{p+1}$\cite[p.~222]{boh-reid}.  By Lemma~\ref{inclusion graph l1}, 
$D_{p+1}$ together with its proper subgroups forms $K_2+\overline{K}_3$ as a subgraph of $\mathcal I(G)$, when $p > 7$.
	If $p=7$, then $S_4$ is a maximal subgroup of $L_2(7)$. Since $A_4$ is a subgroup of $S_4$, we already dealt this case.
	
	\noindent \textbf{Case 4:}  $G \cong Sz(2^q)$. Then  $ Sz(2^q)$ has a subgroup isomorphic to 
	$(\mathbb Z_2)^q, q \geq 3$. By Proposition~\ref{inclusion p1}, $\mathcal I((\mathbb Z_2)^q )$, $q >2$ 
	contains $K_{1,3}$ as a subgraph; $(\mathbb Z_2)^q$ together with its proper subgroups forms 
	$K_2+\overline{K}_3$ as a subgraph of $\mathcal I(G)$. 
	
	Proof follows by putting together all the above cases.
	\end{proof}
Combining the Propositions~\ref{inclusion p1} --~\ref{inclusion graph t11}, we obtain the proof of Theorem~\ref{inclusion graphs 7}.

\end{document}